\documentclass[12pt]{amsart}
\usepackage{a4wide,enumerate,color}
\usepackage{amsmath}
\usepackage{amssymb}
\allowdisplaybreaks

\let\pa\partial
\let\na\nabla
\let\eps\varepsilon
\newcommand{\N}{{\mathbb N}}
\newcommand{\R}{{\mathbb R}}
\newcommand{\C}{{\mathbb C}}
\newcommand{\diver}{\operatorname{div}}

\newcommand{\m}{\vec{\mu}}

\newtheorem{theorem}{Theorem}
\newtheorem{lemma}[theorem]{Lemma}

\newtheorem{remark}[theorem]{Remark}

\begin{document}

\title[Spin-polarized drift-diffusion model]{Large-time asymptotics for a matrix
spin drift-diffusion model}

\author[P. Holzinger]{Philipp Holzinger}
\address{Institute for Analysis and Scientific Computing, Vienna University of
	Technology, Wiedner Hauptstra\ss e 8--10, 1040 Wien, Austria}
\email{philipp.holzinger@tuwien.ac.at}

\author[A. J\"ungel]{Ansgar J\"ungel}
\address{Institute for Analysis and Scientific Computing, Vienna University of
	Technology, Wiedner Hauptstra\ss e 8--10, 1040 Wien, Austria}
\email{juengel@tuwien.ac.at}

\date{\today}

\thanks{The authors acknowledge partial support from
the Austrian Science Fund (FWF), grants P30000, W1245, and F65.}

\begin{abstract}
The large-time asymptotics of the density matrix solving a
drift-diffusion-Poisson model for the spin-polarized electron transport
in semiconductors is proved. The equations are analyzed in a bounded
domain with initial and Dirichlet boundary conditions. If the
relaxation time is sufficiently small and the boundary data is close to
the equilibrium state, the density matrix converges exponentially fast
to the spinless near-equilibrium steady state.
The proof is based on a reformulation of the
matrix-valued cross-diffusion equations using spin-up and spin-down densities
as well as the perpendicular component of the spin-vector density,
which removes the cross-diffusion terms.
Key elements of the proof are time-uniform positive lower and upper bounds
for the spin-up and spin-down densities, derived from the De Giorgi--Moser
iteration method, and estimates of the relative free energy for the
spin-up and spin-down densities.
\end{abstract}



\maketitle


\section{Introduction}

Semiconductor lasers and transistor devices may be improved by taking into account
spin-polarized electron injection. The corresponding semiconductor models should
include the spin effects in an accurate way. A widely used model are the
two-component spin drift-diffusion equations \cite{FMESZ07},
which can be derived for strong spin-orbit coupling from the spinorial
Boltzmann equation in the diffusion limit \cite{ElH14},
describing the dynamics of the spin-up and spin-down electrons.
When the spin-orbit coupling is only moderate, the diffusion limit in the
spinorial Boltzmann equation leads to a matrix spin drift-diffusion model
for the electron density matrix \cite{ElH14,PoNe11}. This model contains much more
information than the two-component model, but the strong coupling between the
four spin components makes the mathematical analysis very challenging.
The existence of global weak solutions was shown in \cite{JNS15}.
In this paper, we investigate the large-time asymptotics of the density matrix
towards a near-equilibrium steady state.

\subsection{Model equations}

We assume that
the dynamics of the (Hermitian) density matrix $N(x,t)\in\C^{2\times 2}$, the
current density matrix $J(x,t)\in\C^{2\times 2}$, and the electric potential $V(x,t)$
is given by the (scaled) matrix equations
\begin{align}
  & \pa_t N - \diver J + \mathrm{i}\gamma[N,\m\cdot\vec{\sigma}]
	= \frac{1}{\tau}\bigg(\frac12\operatorname{tr}(N)\sigma_0 - N\bigg), \label{1.eq1} \\
	& J = DP^{-1/2}(\na N+N\na V)P^{-1/2}, \label{1.eq2} \\
	& -\lambda^2\Delta V = \operatorname{tr}(N) - g(x) \quad\mbox{in }\Omega,\ t>0,
	\label{1.eq3}
\end{align}
where $[A,B]=AB-BA$ is the commutator for matrices $A$ and $B$. The (scaled)
physical parameters are the strength of the pseudo-exchange field $\gamma>0$,
the normalized precession vector $\m =(\mu_1,\mu_2,\mu_3)\in\R^3$,
the spin-flip relaxation time $\tau>0$, the diffusion constant $D>0$,
the Debye length $\lambda>0$, and the doping concentration $g(x)$.
Equation \eqref{1.eq3} is the Poisson equation for the electric potential
\cite{Jue09}. The precession vector plays the role
of the local direction of the magnetization in the ferromagnet, and we assume
that it is constant. This assumption is crucial for our analysis.
Furthermore, $P=\sigma_0+p\m \cdot\vec\sigma
=\sigma_0+p(\mu_1\sigma_1+\mu_2\sigma_2+\mu_3\sigma_3)$
is the matrix of spin polarization of the scattering rates,
where $\sigma_0$ is the unit matrix in $\C^{2\times 2}$, $\vec\sigma=(\sigma_1,
\sigma_2,\sigma_3)$ is the vector of Pauli matrices, and $p\in[0,1)$ represents
the spin polarization. The number $\mathrm{i}$ is the complex unit, and $\mbox{tr}(N)$
denotes the trace of the matrix $N$.
The commutator $[N,\m \cdot\vec\sigma]$
models the precession of the spin polarization. The right-hand side in
\eqref{1.eq1} describes the spin-flip relaxation of the spin density to the
(spinless) equilibrium state.

Equations \eqref{1.eq1}--\eqref{1.eq3} are solved in the bounded domain
$\Omega\subset\R^3$ with time $t>0$ and are supplemented with
the boundary and initial conditions
\begin{equation}
  N = \frac12n_D\sigma_0,\ V=V_D\quad\mbox{on }\pa\Omega,\ t>0, \quad
  N(0) = N^0 \quad\mbox{in }\Omega. \label{1.bic}
\end{equation}
This means that no spin effects occur on the boundary.
For simplicity, we choose time-independent boundary data; see \cite{ZaJu13}
for boundary data depending on time. Mixed Diri\-chlet--Neumann boundary
conditions may be also considered as long as they allow for $W^{2,q_0}(\Omega)$
elliptic regularity results, which restricts the geometry of $\pa\Omega$.
Therefore, we have chosen pure Dirichlet boundary data as in \cite{JNS15}.

The density matrix $N$ can be expressed in terms of the Pauli matrix according to
$N = \frac12 n_0\sigma_0 + \vec n\cdot\vec\sigma$ and $\vec n=(n_1,n_2,n_3)$ is called
the spin-vector density. Model \eqref{1.eq1}--\eqref{1.eq2}, written in the
four variables $n_0,\ldots,n_3$, is a cross-diffusion system with the constant
diffusion matrix
$$
  \frac{D}{1-p^2}\begin{pmatrix}
	1 & -p\m ^\top \\ -p\m  & \eta\mathbb{I} + (1-\eta)\m \otimes\m
	\end{pmatrix} \in \R^{4\times 4},
$$
where $\mathbb{I}$ is the unit matrix in $\R^{3\times 3}$. Although this matrix
is symmetric and positive definite, the strong coupling complicates the analysis
of system \eqref{1.eq1}--\eqref{1.eq2}, because maximum principle arguments
and other standard tools cannot be (easily) applied.

The spin polarization matrix couples the charge and spin components of the
electrons. If $p=0$, we recover the classical Van-Roosbroeck drift-diffusion
equations for the electron charge density $n_0$ \cite{Mar86,Van50},
\begin{equation}\label{1.dd}
  \pa_t n_0 - D_p\diver J_0 = 0, \quad J_0 = \na n_0+n_0\na V, \quad
	-\lambda^2\Delta V = n_0 - g(x),
\end{equation}
where $J_0$ denotes the charge current density and $D_p=2D/(1-p^2)$.
The boundary conditions are $n_0=n_D$ and $V=V_D$ on $\pa\Omega$
and the initial condition is $n_0(0)=n_0^0$ in $\Omega$,
where $N^0 = \frac12 n_0^0\sigma_0 + \vec n^0\cdot\vec\sigma$. Another special case is
given by the two-component spin drift-diffusion model. The spin-up
and spin-down densities $n_\pm = \frac12 n_0\pm \vec n\cdot\m $, respectively,
satisfy the equations
\begin{align}
  & \pa_t n_+ - \diver\big(D_+(\na n_+ + n_+\na V)\big) = \frac{1}{2\tau}(n_- - n_+),
	\label{1.nplus} \\
	& \pa_t n_- - \diver\big(D_-(\na n_- + n_-\na V)\big) = \frac{1}{2\tau}(n_+ - n_-),
	\label{1.nminus} \\
	& n_\pm = \frac{n_D}{2},\quad V=V_D\quad\mbox{on }\pa\Omega, \quad
	n_\pm(0)=\frac12 n_0^0\pm\vec n^0\cdot\m \quad\mbox{in }\Omega, \label{1.bic2}
\end{align}
where $D_\pm=D/(1\pm p)$.
These equations are weakly coupled through the relaxation term.

Model \eqref{1.eq1}--\eqref{1.eq2} was derived in \cite{PoNe11}
from a matrix Boltzmann equation in the diffusion limit. The scattering operator
in the Boltzmann model is assumed to consist of a dominant collision operator from
the Stone model and a spin-flip relaxation operator. When the scattering rate
in the Stone model is smooth and invariant under isometric transformations,
the diffusion $D$ can be identified with a positive number \cite[Prop.~1]{Pou91}.

\subsection{State of the art}

The first result on the global existence of solutions to the Van-Roosbroeck
equations \eqref{1.dd} (for electrons and positively charged holes)
was proved by Mock \cite{Moc74}. He showed in \cite{Moc75} that
the solution decays exponentially fast to the equilibrium state provided that the initial
data is sufficiently close to the equilibrium. These results were generalized
under physically more realistic assumptions on the boundary by Gajewski
\cite{Gaj85} and Gajewski and Gr\"oger \cite{GaGr86,GaGr89}.
Further large-time asymptotics can be found in \cite{AMT00} for the whole-space
problem and in \cite{BMV04}, where the diffusion constant was replaced by a
diffusion matrix. Moreover, in \cite{DiWu08},
the stability of the solutions in Wasserstein spaces was investigated.

Convergence rates of the whole-space solutions to their self-similar profile were
investigated intensively in the literature. In \cite{BiDo00}, the relative
free energy allowed the authors to prove the self-similar asymptotics in the
$L^1(\R^d)$ norm. The results were improved in \cite{KoKa08}, showing optimal
$L^p(\R^d)$ decay estimates.
The asymptotic profile to drift-diffusion-Poisson equations
with fractional diffusion was analyzed in \cite{OgYa09,Yam12}.

Concerning drift-diffusion models for the spin-polarized electron transport,
there are only few mathematical results.
The stationary two-component drift-diffusion model
\eqref{1.nplus}--\eqref{1.nminus}
was analyzed in \cite{GlGa10}, while the transient equations were investigated in
\cite{Gli08}. In particular, Glitzky proved in \cite{Gli08} the
exponential decay to equilibrium. An existence analysis for a diffusion model for
the spin accumulation with fixed electron current but non-constant magnetization
was proved in \cite{PuGu10} in one space dimension and in \cite{GaWa07,GaWa15}
for three space dimensions.

Also quantum spin diffusion models have been considered. For instance,
in \cite{ZaJu13}, the large-time asymptotics for a simple spin drift-diffusion
system for quantum electron transport in graphene was studied.
A more general quantum spin drift-diffusion model was derived in \cite{BaMe10},
with numerical experiments in \cite{BMNP15}. Numerical simulations
for diffusion models for the spin accumulation, coupled with the
Landau--Lifshitz--Gilbert equation, can be found in \cite{ARBVHPS15,RAHSP16}.
For spin transport models in superlattices, we refer, for instance, to \cite{BBA10}.

The existence of global weak solutions to the matrix spin
drift-diffusion model \eqref{1.eq1}--\eqref{1.bic}
was shown in \cite{JNS15} with constant
precession vector and in \cite{Zam14} with non-constant precession vector
but assuming velocity saturation. Under the condition that the (thermal)
equilibrium density is sufficiently small, the exponential decay to equilibrium
was proved in \cite{Zam14}.
An implicit Euler finite-volume scheme that preserves some of the features of
the continuous model was analyzed in \cite{CJS16}. The numerical
results of \cite{CJS16} indicate that the relative free energy is
decaying with exponential rate, but no analytical proof was given.

In this paper, we prove that the solution $(N(t),V(t))$
to \eqref{1.eq1}--\eqref{1.bic} converges exponentially
fast to a steady state $(\frac12n_\infty\sigma_0,V_\infty)$, solving the stationary
spinless drift-diffusion-Poisson equations
\begin{align}
  & \diver(\na n_\infty + n_\infty\na V_\infty) = 0, \quad
	-\lambda^2 \Delta V_\infty = n_\infty - g(x) \quad\mbox{in }\Omega \label{1.stat1} \\
	& n_\infty = n_D, \quad V=V_D\quad\mbox{on }\pa\Omega, \label{1.stat2}
\end{align}
under the condition that the boundary data is close to the (thermal)
equilibrium state, defined by $\log n_D+V_D=0$ on $\pa\Omega$.
Compared to \cite{Zam14}, where $\|n_\infty\|_{L\infty(\Omega)}\ll 1$ is needed,
our smallness assumption is physically reasonable;
see the discussion in the following subsection.

\subsection{Main result and key ideas}

Our main result is as follows.

\begin{theorem}[Exponential time decay]\label{thm.time}
Let $T>0$ and let $\Omega\subset\R^3$ be a bounded domain with $\pa\Omega\in C^{1,1}$.
Furthermore, let $0<m_*<1$, $\lambda>0$, $\gamma>0$, $D>0$, $0\le p<1$, $q_0>3$,
and $\m \in\R^3$ with $|\m |=1$.
The data satisfies $g\in L^\infty(\Omega)$, $g\ge 0$ in $\Omega$, and
\begin{align*}
  & n_D,\ V_D\in W^{2,q_0}(\Omega), \quad n_D\ge m_*>0\mbox{ on }\pa\Omega, \\
  & n_0^0,\, \vec{n}^0\cdot\m \in L^\infty(\Omega), \quad
  \frac12 n_0^0\pm\vec{n}^0\cdot\m \ge \frac{m_*}{2} > 0.
\end{align*}
Let $\phi_D:=\log n_D+V_D$.
Then there exist $\kappa>0$, $C_0>0$, and $\delta>0$ such that if
$\|\phi_D\|_{W^{2,q_0}(\Omega)}\le\delta$,
$$
  \|n_\pm(t)-\tfrac12 n_\infty\|_{L^2(\Omega)}
	+ \|V(t)-V_\infty\|_{H^1(\Omega)}\le C_0e^{-\kappa t}, \quad t>0,
$$
where $(n_\infty,V_\infty)$ is the weak solution to \eqref{1.stat1}--\eqref{1.stat2}.
Furthermore, there exists $\tau_0>0$ such that if $0<\tau\le\tau_0$ then
$$
  \|N(t)-\tfrac12 n_\infty\sigma_0\|_{L^2(\Omega;\C^{2\times 2})}
	\le C_0^*e^{-\kappa^* t}, \quad t>0,
$$
and $C_0$, $C_0^*>0$ depend on the initial relative free energy $H(0)$
(see \eqref{1.H} below).
\end{theorem}

The smallness condition on $\phi_D$ means that the system is close to equilibrium,
as $\phi_D=0$ characterizes the (thermal) equilibrium state. Since the stationary
drift-diffusion equations \eqref{1.stat1} may possess multiple solutions if
$\phi_D$ is large in a certain sense \cite{Ala95,Moc75}, the condition on
$\phi_D$ is not surprising. The smallness condition on the relaxation time,
however, seems to be purely technical. It is needed to estimate the drift part when
we derive $L^2(\Omega)$ bounds for the perpendicular component of $\vec n$.
If an entropy structure exists for the equation
for $\vec n$, we expect that this condition can be avoided but currently, such
a structure is not clear; see \cite[Remarks 3.1--3.2]{JNS15}. If the
initial spin-vector density is parallel to the precession vector, we are able
to remove the smallness condition on $\tau$; see Remark \ref{rem.tau}.
We show in Remark \ref{rem.smalltau} that, independently of the initial
spin-vector density, the smallness condition is satisfied
in a certain physical regime.

The analysis of the asymptotic behavior of the solutions to
the Van-Roosbroeck drift-diffusion system
\eqref{1.dd} and the two-component system \eqref{1.nplus}--\eqref{1.nminus}
is based on the observation that the relative free energy, consisting of the
internal and electric energies, is a Lyapunov functional along the
solutions and that the energy dissipation can be bounded from below in terms
of the relative free energy itself. The strong coupling of \eqref{1.eq1}-\eqref{1.eq2}
prohibits this approach. Indeed, we showed in \cite[Section 3]{JNS15} that
the relative free energy associated to \eqref{1.eq1}-\eqref{1.eq2},
consisting of the von Neumann energy and the electric energy, is nonincreasing in time
only in very particular cases.

Our idea is the observation that the matrix system \eqref{1.eq1}--\eqref{1.eq2}
can be reformulated as drift-diffusion-type equations in terms of certain
projections of the density matrix relative to the precession vector. This idea
was already used in \cite{JNS15} for the existence analysis. The reformulation
removes the cross-diffusion terms, which allows us to apply the techniques
of Gajewski and Gr\"oger \cite{GaGr89} used for the Van-Roosbroeck model.
This idea only works if the precession vector $\m $ is constant.
A non-constant vector
$\m $ (solving the Landau--Lifshitz--Gilbert equation) was considered
in \cite{ZaJu16}, but this spin model is simplified and no large-time asymptotics
was proved.

More precisely, we decompose the density matrix
$N=\frac12 n_0\sigma_0 + \vec n\cdot\vec\sigma$ and
$N^0=\frac12 n_0^0\sigma_0 + \vec n^0\cdot\vec\sigma$.
Then the spin-up and spin-down
densities $n_\pm = \frac12 n_0\pm \vec n\cdot\m $ solve
\eqref{1.nplus}--\eqref{1.bic2}.
The information on $n_\pm$ is not sufficient to recover the density matrix.
Therefore, we also consider the perpendicular component of
$\vec n$ with respect to $\m $, $\vec n_\perp = \vec n-(\vec n\cdot\m )\m $,
which solves
\begin{equation}
  \pa_t\vec n_\perp - \diver\bigg(\frac{D}{\eta}(\na\vec n_\perp
	+ \vec n_\perp\na V)\bigg) - 2\gamma(\vec n_\perp\times\m )
	= -\frac{\vec n_\perp}{\tau}, \label{1.nperp}
\end{equation}
where $\eta=\sqrt{1-p^2}$, with
the boundary and initial conditions $\vec n_\perp=0$ on $\pa\Omega$ and
$\vec n_\perp = \vec n^0 - (\vec n^0\cdot\m )\m $.
The density matrix can be reconstructed from $(n_+,n_-,\vec n_\perp)$ by setting
$n_0=n_+ + n_-$ and $\vec n=\vec n_\perp + (\vec n\cdot\m )\m
= \vec n_\perp + \frac12(n_+ - n_-)\m $.

A key element of the proof is the derivation of a uniform positive lower
bound for $n_\pm$. This is shown by using the De Giorgi--Moser iteration method
inspired by the proof of \cite[Lemma 3.6]{GaGr89}. More precisely, we choose
the test functions $e^{t}w_\pm^{q-1}/n_\pm$ in \eqref{1.nplus} and \eqref{1.nminus},
respectively, where $w_\pm = -\min\{0,\log n_\pm+m\}$ with $m>0$, $q\in\N$, and
pass to the limit $q\to\infty$, leading to $\|w_\pm(t)\|_{L^\infty(\Omega)}
\le K$ and consequently to the desired bound $w_\pm(t)\ge e^{-m-K}$ in $\Omega$.
Second, we calculate the time derivative of the free energy
\begin{equation}\label{1.H}
  H(t) = \int_\Omega\bigg(h(n_+|n_\infty) + h(n_-|n_\infty)
	+ \frac{\lambda^2}{2}|\na(V-V_\infty)|^2\bigg)(t)dx,
\end{equation}
where $h(n_\pm|n_\infty)=n_\pm\log(2n_\pm/n_\infty)-n_\pm+\tfrac12 n_\infty$,
leading to the free energy inequality
\begin{align}
  \frac{dH}{dt} &+ C_1\int_\Omega\big(n_+|\na(\phi_+-\phi_D)|^2
	+ n_-|\na(\phi_--\phi_D)|^2\big) \nonumber \\
	&\le C_2\|\na\phi_D\|_{L^\infty(\Omega)}^2
	\int_\Omega\big((n_+-\tfrac12 n_\infty)^2
	+ (n_--\tfrac12 n_\infty)^2\big)dx, \label{1.dHdt}
\end{align}
where $\phi_\pm=\log n_\pm+V$ are the electrochemical potentials
and $C_1>0$ and $C_2>0$ are some constants independent of the solution
and independent of time. The right-hand side can be estimated, up to a factor,
by the free energy $H$ times $\|\na\phi_D\|_{L^\infty(\Omega)}^2$. Furthermore,
using the time-uniform positive lower bound for $n_\pm$, the energy dissipation
(the second term on the left-hand side of \eqref{1.dHdt}) is bounded from below
by $H$, up to a factor. Therefore, if $\|\na\phi_D\|_{L^\infty(\Omega)}\le\delta$,
\eqref{1.dHdt} becomes, for some time-independent constants $C_3>0$ and $C_4>0$,
$$
  \frac{dH}{dt} + (C_3-C_4\delta^2)H \le 0.
$$
Choosing $\delta^2<C_3/C_4$, the Gronwall inequality implies the exponential decay
with respect to the free energy and, as a consequence,
in the $L^2(\Omega)$ norm of $n_\pm-n_\infty$ with rate $\kappa:=C_3-C_4\delta^2>0$.

Third, we prove the time decay of $\vec n_\perp$.
Since we are not aware of an entropy structure for \eqref{1.nperp},
we rely on $L^2(\Omega)$ estimates. This means that we use the test function
$\vec n$ in the weak formulation of \eqref{1.nperp} such that the term
$(\vec n\times\m )\cdot\vec n$ vanishes. However, in order to handle
the term coming from the doping concentration, we need a smallness
condition on the relaxation time $\tau>0$. Such a condition is not needed in the
Van-Rooosbroeck model.

The paper is organized as follows. The stationary equations are studied in
Section \ref{sec.steady}. In Section \ref{sec.unif}, we prove the lower and
upper uniform bounds for $n_\pm$, the entropy inequality, and some
bounds for the free energy and energy dissipation. Theorem \ref{thm.time}
is proved in Section \ref{sec.proof}. In the appendix, we prove a uniform $L^\infty$
bound for any function that satisfies an iterative inequality using the
De Giorgi--Moser method.


\section{The stationary equations}\label{sec.steady}

The existence of weak solutions to the stationary drift-diffusion problem
\begin{align}
  & \diver(\na n_\infty + n_\infty\na V_\infty) = 0, \quad
	-\lambda^2\Delta V_\infty = n_\infty - g(x)\quad\mbox{in }\Omega, \label{2.steady} \\
	& n_\infty=n_D, \quad V_\infty=V_D\quad\mbox{on }\pa\Omega \label{2.bc}
\end{align}
with data satisfying the assumptions in Theorem \ref{thm.time}
is well known; see \cite[Theorem 3.2.1]{Mar86}. The solution satisfies
$n_\infty$, $V_\infty\in H^1(\Omega)\cap L^\infty(\Omega)$ and
\begin{equation}\label{2.minfty}
  0<m_\infty\le n_\infty\le M_\infty\quad\mbox{in }\Omega
\end{equation}
for some $m_\infty$, $M_\infty>0$. Note that we cannot expect uniqueness of
weak solutions in general, since there are devices (thyristors) that allow
for multiple physical stationary solutions. However, uniqueness can be
expected for data sufficiently close to the (thermal)
equilibrium state \cite{Ala95,Moc82}.
We call a solution to \eqref{2.steady}--\eqref{2.bc} a (thermal) equilibrium state if
the electrochemical potential
$\phi_\infty:= \log n_\infty+V_\infty$ vanishes in $\Omega$.
This state needs the compatibility condition $\phi_D:=\log n_D+V_D=0$ on $\pa\Omega$.

The following lemma provides some a priori estimates for $(n_\infty,V_\infty)$ and
shows that the current density $J_\infty:=n_\infty\na\phi_\infty$ is arbitrarily
small in the $L^\infty(\Omega)$ norm if the boundary data is sufficiently
close to the equilibrium state $\phi_D=0$ in the $W^{2,q_0}(\Omega)$ sense.

\begin{lemma}[A priori estimates]\label{lem.stat}
Under the assumptions of Theorem \ref{thm.time}, there
exists a constant $C_\infty>0$ independent of $(n_\infty,V_\infty)$ such that
$$
  \|\na\phi_\infty\|_{L^\infty(\Omega)}	\le C_\infty\|\phi_D\|_{W^{2,q_0}(\Omega)}.
$$
\end{lemma}

\begin{proof}
Since $n_\infty-g(x)\in L^\infty(\Omega)$, elliptic regularity yields
$V_\infty\in W^{2,q_0}(\Omega)$, and the $W^{2,q_0}(\Omega)$ norm of $V_\infty$
depends on $M_\infty$, $\|V_D\|_{W^{2,q_0}(\Omega)}$, and $\|g\|_{L^\infty(\Omega)}$.
Using the test function $n_\infty-n_D$ in the weak formulation of the first
equation in \eqref{2.steady}, we find that
\begin{align*}
  \int_\Omega|\na(n_\infty-n_D)|^2 dx
	&= -\int_\Omega\na n_D\cdot\na(n_\infty-n_D)dx
	- \int_\Omega n_\infty\na V_\infty\cdot\na(n_\infty-n_D)dx \\
	&\le \big(\|\na n_D\|_{L^2(\Omega)} + M_\infty\|\na V_\infty\|_{L^2(\Omega)}\big)
	\|\na(n_\infty-n_D)\|_{L^2(\Omega)}.
\end{align*}
This implies that
\begin{align*}
  \|\na n_\infty\|_{L^2(\Omega)}
	&\le 2\|\na n_D\|_{L^2(\Omega)}
	+ M_\infty\|\na V_\infty\|_{L^2(\Omega)} \\
	&\le C\big(1 + \|\na n_D\|_{L^2(\Omega)} + \|\na V_D\|_{L^2(\Omega)}\big).
\end{align*}
Since $q_0>3$, we have $W^{2,q_0}(\Omega)\hookrightarrow W^{1,\infty}(\Omega)$.
Thus, $b:=\na V_\infty\in L^\infty(\Omega)$ and elliptic regularity for
$$
  \Delta n_\infty + b\cdot\na n_\infty = \lambda^{-2}n_\infty(n_\infty-g(x))
	\in L^\infty(\Omega)
$$
shows that $n_\infty\in W^{2,q_0}(\Omega)$ with an a priori bound depending on the
norms $\|n_D\|_{W^{2,q_0}(\Omega)}$ and $\|n_\infty\|_{H^1(\Omega)}$. Summarizing,
$$
  \|n_\infty\|_{W^{2,q_0}(\Omega)} + \|V_\infty\|_{W^{2,q_0}(\Omega)} \le C.
$$

It holds that $W^{2,q_0}(\Omega)\hookrightarrow C^{0,\alpha}(\overline\Omega)$
for all $0<\alpha<1$. Hence, $n_\infty\in C^{0,\alpha}(\overline\Omega)$.
The first equation in \eqref{2.steady} can be formulated as
$$
  \diver\big(n_\infty\na(\phi_\infty-\phi_D)\big)
	= -\diver(n_\infty\na\phi_D)\in L^{q_0}(\Omega),
$$
which shows that, by elliptic regularity again,
$$
  \|\phi_\infty-\phi_D\|_{W^{2,q_0}(\Omega)}
	\le C(m_\infty,M_\infty)\|\phi_D\|_{W^{2,q_0}(\Omega)}
$$
and, in view of the continuous embedding
$W^{2,q_0}(\Omega)\hookrightarrow W^{1,\infty}(\Omega)$,
$$
  \|\na\phi_\infty\|_{L^\infty(\Omega)}
	\le C\|\phi_\infty\|_{W^{2,q_0}(\Omega)}
	\le C\|\phi_D\|_{W^{2,q_0}(\Omega)}.
$$
This finishes the proof.
\end{proof}


\section{Uniform estimates}\label{sec.unif}

In this section, we prove some a priori estimates that are uniform in time.
A uniform upper bound for $n_\pm$ was already shown
in \cite[Theorem 1.1]{JNS15}. We present a slightly shorter proof than
that one presented in \cite{JNS15}.

\begin{lemma}[Uniform upper bound for $n_\pm$]\label{lem.upper}
Introduce
$$
  M = \max\bigg\{\sup_{\pa\Omega}n_D,\sup_\Omega n_0^0,\sup_\Omega g\bigg\}.
$$
Then
$$
  n_\pm(t) \le M\quad\mbox{in }\Omega,\ t>0.
$$
\end{lemma}

\begin{proof}
We prove that $n_0\le M$, where $n_0=n_++n_-$ is a weak solution to
$$
  \pa_t n_0 = D_p\diver(\na n_0+n_0\na V)\quad\mbox{in }\Omega, \quad
	n_0=n_D\quad\mbox{on }\pa\Omega, \quad n_0(0)=n_0^0\quad\mbox{in }\Omega,
$$
where $D_p=2D/(1-p^2)$.
Then $0\le n_\pm\le n_0\le M$ in $\Omega$. We choose the test function
$(n_0-M)^+ = \max\{0,n_0-M\}$ in the weak formulation of the previous equation,
using $(n_0-M)^+=0$ on $\pa\Omega$ and $(n_0(0)-M)^+=0$ in $\Omega$. Then
\begin{align*}
  \frac12&\int_\Omega|(n_0(t)-M)^+|^2 dx
	+ D_p\int_0^t\int_\Omega|\na (n_0-M)^+|^2 dxds \\
	&= -D_p\int_0^t\int_\Omega(n_0-M)\na V\cdot\na(n_0-M)^+ dxds
	- D_pM\int_0^t\int_\Omega\na V\cdot\na (n_0-M)^+ dxds.
\end{align*}
Writing $(n_0-M)\na V\cdot\na(n_0-M)^+ = \na V\cdot\frac12\na[(n_0-M)^+]^2$, integrating by
parts, and using the Poisson equation leads to
\begin{align*}
  \frac12\int_\Omega|(n_0(t)-M)^+|^2 dx
	&\le -\frac{D_p}{2\lambda^2}\int_0^t\int_\Omega(n_0-g(x))[(n_0-M)^+]^2 dxds \\
	&\phantom{xx}{}- \frac{D_pM}{\lambda^2}\int_0^t\int_\Omega(n_0-g(x))(n_0-M)^+ dxds
	\le 0,
\end{align*}
since $n_0-g(x)>M-g(x)\ge 0$ on the set $\{n_0>M\}$. We conclude that
$(n_0(t)-M)^+=0$ and $n_0(t)\le M$ in $\Omega$, $t>0$.
\end{proof}

\begin{lemma}[Uniform positive lower bound for $n_\pm$]\label{lem.lower}
There exists $m>0$ such that for all $t>0$,
$$
  n_\pm(t) \ge m > 0 \quad\mbox{in }\Omega.
$$
\end{lemma}

\begin{proof}
We show first that $n_\pm$ is strictly positive with a lower bound that depends on
time. For this, use the test functions $(n_\pm-m^*(t))^-=\min\{0,n_\pm-m^*(t)\}$,
where $m^*(t)=m_0e^{-\mu t}$, $\mu=2\lambda^{-2}D_-M$, and
$$
  m_0 = \min\bigg\{\inf_{\pa\Omega}\frac{n_D}{2},\inf_\Omega\bigg(
	\frac12n_0^0+\vec n^0\cdot\m \bigg)\bigg\} > 0,
$$
in \eqref{1.nplus}, \eqref{1.nminus}, respectively,
and add both equations. Proceeding similarly as in the proof of Lemma \ref{lem.upper},
we obtain
\begin{align*}
  \frac12 & \int_\Omega\big((n_+-m^*)^-(t)^2 + (n_--m^*)^-(t)^2\big)dx \\
	&\phantom{xx}{}+ D_+\int_0^t\int_\Omega|\na(n_+-m^*)^-|^2 dxds
	+ D_-\int_0^t\int_\Omega|\na(n_--m^*)^-|^2 dxds \\
	&= -\frac{1}{2\tau}\int_0^t\int_\Omega(n_+ - n_-)\big((n_+-m^*)^- - (n_--m^*)^-\big)
	 dxds \\
	&\phantom{xx}{} - \frac{D_+}{2}\int_0^t\int_\Omega\na [(n_+ - m^*)^-]^2
	\cdot\na V dxds
	- \frac{D_-}{2}\int_0^t\int_\Omega\na [(n_- - m^*)^-]^2\cdot\na V dxds \\
	&\phantom{xx}{}- D_+\int_0^t m^*(s)\int_\Omega \na(n_+ - m^*)^-\cdot\na V dxds \\
	&\phantom{xx}{}- D_-\int_0^t m^*(s)\int_\Omega \na(n_- - m^*)^-\cdot\na V dxds \\
	&\phantom{xx}{}+ \mu\int_0^t m^*(s)\int_\Omega\big((n_+-m^*)^- + (n_--m^*)^-
	\big)dxds.
\end{align*}
The first term on the right-hand side is nonpositive since $z\mapsto(z-m^*)^-$ is
monotone. For the remaining terms, we use the Poisson equation and the estimate
$n_0=n_++n_-\le 2M$:
\begin{align*}
  \frac12 & \int_\Omega\big((n_+-m^*)^-(t)^2 + (n_--m^*)^-(t)^2\big)dx \\
	&\le -\frac{D_+}{2\lambda^2}\int_0^t\int_\Omega|(n_+ - m^*)^-|^2(n_0-g(x))dxds \\
	&\phantom{xx}{} - \frac{D_-}{2\lambda^2}
	\int_0^t\int_\Omega|(n_- - m^*)^-|^2(n_0-g(x))dxds \\
	&\phantom{xx}{}- \frac{D_+}{\lambda^2}\int_0^t m^*(s)\int_\Omega (n_+ - m^*)^-
  (n_0 - g(x)) dxds \\
	&\phantom{xx}{}- \frac{D_-}{\lambda^2}\int_0^t m^*(s)\int_\Omega (n_- - m^*)^-
  (n_0 - g(x)) dxds \\
	&\phantom{xx}{}+ \mu\int_0^t m^*(s)\int_\Omega
	\big((n_+-m^*)^- + (n_--m^*)^-\big)dxds \\
  &\le \frac{D_-}{2\lambda^2}\|g\|_{L^\infty(\Omega)}\int_0^t\int_\Omega
	\big(|(n_+ - m^*)^-|^2 + |(n_- - m^*)^-|^2\big)dxds \\
	&\phantom{xx}{}
	- \frac{D_+}{\lambda^2}\int_0^t m^*(s)\int_\Omega(n_+ - m^*)^-
	\bigg(2M-\frac {\lambda^2}{D_+}\mu\bigg)dxds \\
	&\phantom{xx}{}
	- \frac{D_-}{\lambda^2}\int_0^t m^*(s)\int_\Omega(n_- - m^*)^-
	\bigg(2M-\frac {\lambda^2}{D_-}\mu\bigg)dxds \\
	&\le \frac{D_-}{2\lambda^2}\|g\|_{L^\infty(\Omega)}\int_0^t\int_\Omega
	\big(|(n_+ - m^*)^-|^2 + |(n_- - m^*)^-|^2\big)dxds.
\end{align*}
In the last inequality, we used $2M-\lambda^2D_\pm^{-1}\mu\le 0$. By Gronwall's
lemma, this shows that $n_\pm\ge m^*(t)>0$ in $\Omega$.

In the second step, we prove that $n_\pm$ is strictly positive uniformly in time.
The idea is to use the Di Giorgi--Moser iteration method similarly as
in the proof of Lemma 3.6 in \cite{GaGr89}.
We set $w_\pm=-(\log n_\pm+m)^-\in L^2(0,T;H^1(\Omega))$  and take the
test function $e^{t}w_\pm^{q-1}/n_\pm$ in \eqref{1.nplus},
\eqref{1.nminus}, respectively, where $0<-\log(m_*/2)<m<1$ and $q\in\N,q\ge2$.
Because of the previous step which ensures that $n_\pm>0$, this test function 
is well defined. Moreover, $\log(n_D/2)+m\ge \log(m_*/2)+m\ge 0$
and $\log(n_0^0/2\pm\vec n^0\cdot\m )+m\ge \log(m_*/2)+m\ge 0$ such that
$w_\pm=0$ on $\pa\Omega$ and $w_\pm(0)=0$ in $\Omega$.
Formally, we compute $\pa_t(e^{t}w_\pm^q) - e^{t}w_\pm^q
= -qe^{t}w_\pm^{q-1}n_\pm^{-1}\pa_t n_\pm$.
Therefore, integrating this identity formally over $\Omega$ and $(0,t)$ and
using \eqref{1.nplus}--\eqref{1.nminus},
\begin{align*}
  \int_\Omega& e^{t}(w_+^q(t) + w_-^q(t))dx
	- \int_0^t\int_\Omega e^{s}(w_+^q + w_-^q)dxds \\
	&= -q\int_0^t e^{s}\bigg(\bigg\langle\pa_t n_+,\frac{w_+^{q-1}}{n_+}\bigg\rangle
	+ \bigg\langle\pa_t n_-,\frac{w_-^{q-1}}{n_-}\bigg\rangle\bigg)ds \\
	&= \frac{q}{2\tau}\int_0^t e^{s}\int_\Omega(n_+ - n_-)
	\bigg(\frac{w_+^{q-1}}{n_+} - \frac{w_-^{q-1}}{n_-}\bigg)dxds \\
	&\phantom{xx}{}-D_+ q\int_0^t e^{s}\int_\Omega(\na n_+ + n_+\na V)
	\cdot((q-1)w_+^{q-2} + w_+^{q-1})\frac{\na n_+}{n_+^2}dxds \\
	&\phantom{xx}{}- D_- q\int_0^t e^{s}\int_\Omega(\na n_- + n_-\na V)
	\cdot((q-1)w_-^{q-2} + w_-^{q-1})\frac{\na n_-}{n_-^2}dxds,
\end{align*}
where $\langle\cdot,\cdot\rangle$ is the duality product of $H^{-1}(\Omega)$
and $H_0^1(\Omega)$. The computation can be made rigorous by a density argument;
see \cite[(5.18)]{Jue95} for a similar statement.
Since $z\mapsto (-(\log z+m)^-)^{q-1}/z$ is nonincreasing for $z>0$, the first term
on the right-hand side is nonpositive, giving
\begin{align*}
  \int_\Omega& e^{t}(w_+^q(t) + w_-^q(t))dx
	- \int_0^t\int_\Omega e^{s}(w_+^q + w_-^q)dxds \\
	&\le -D_+q\int_0^t e^{s}\int_\Omega ((q-1)w_+^{q-2}+w_+^{q-1})
	\big(|\na w_+|^2 - \na V\cdot\na w_+\big) dxds \\
	&\phantom{xx}{}
	- D_-q\int_0^t e^{s}\int_\Omega ((q-1)w_-^{q-2}+w_-^{q-1})
	\big(|\na w_-|^2 - \na V\cdot\na w_-\big) dxds.
\end{align*}
Taking into account the Poisson equation and the inequalities $D_+\le D_-$,
$w_\pm\ge 0$, and $n_0\le 2M$, this becomes
\begin{align}
  \int_\Omega& e^{t}(w_+^q(t) + w_-^q(t))dx
	- \int_0^t\int_\Omega e^{s}(w_+^q + w_-^q)dxds \nonumber \\
	&\phantom{xx}{}+\frac{4D_+(q-1)}{q}\int_0^t e^{s}\int_\Omega
	\big(|\na w_+^{q/2}|^2 + |\na w_-^{q/2}|^2\big)dxds \nonumber \\
	&\phantom{xx}{}+ \frac{4D_+q}{(q+1)^2}\int_0^t e^{s}
	\int_\Omega\big(|\na w_+^{(q+1)/2}|^2 + |\na w_-^{(q+1)/2}|^2\big)dxds \nonumber \\
	&\le \frac{D_+q}{\lambda^2}\int_0^t e^{s}\int_\Omega
	\bigg(w_+^{q-1} + \frac{1}{q}w_+^q\bigg)(n_0-g(x))dxds \nonumber \\
	&\phantom{xx}{}+ \frac{D_-q}{\lambda^2}\int_0^t e^{s}\int_\Omega
	\bigg(w_-^{q-1} + \frac{1}{q}w_-^q\bigg)(n_0-g(x))dxds \nonumber \\
	&\le \frac{2D_+M}{\lambda^2}\int_0^t e^{s}\int_\Omega
	(qw_+^{q-1} + w_+^q)dxds + \frac{2D_-M}{\lambda^2}\int_0^t e^{s}\int_\Omega
	(qw_-^{q-1} + w_-^q)dxds \nonumber \\
	&\le \frac{2D_-Mq}{\lambda^2}\int_0^t e^{s}\int_\Omega(w_+^q+w_-^q)dxds
	+ \frac{4D_-M}{\lambda^2}\int_0^t e^{s}\int_\Omega dxds. \label{3.wq}
\end{align}
In the last inequality, we used Young's inequality:
$qw_\pm^{q-1}\le (q-1)w_\pm^q + 1$. We infer that
\begin{align*}
  \int_\Omega & e^t(w_+^q(t)+w_-^q(t))dx
	+ K_0\int_0^t\int_\Omega e^{s}\big(|\na w_+^{q/2}|^2 + |\na w_-^{q/2}|^2\big)dxds \\
	&\le K_1q\int_0^t\int_\Omega e^{s}(w_+^q + w_-^q)dxds
	+ K_2 e^{t}
\end{align*}
for some constants $K_0$, $K_1$, $K_2>0$ which are independent of $q$ and time.

Lemma \ref{lem.iter} in the appendix shows that $w_\pm$ is bounded in $L^\infty$
with a constant which depends on the $L^\infty(0,T;L^1(\Omega))$ norm of $w_\pm$.
Therefore, it remains to estimate $w_\pm$ in this norm. To this end, we
take $q=2$ in \eqref{3.wq}:
\begin{align*}
  \int_\Omega & e^t(w_+^2(t) + w_-^2(t))dx
	- \int_0^t e^s\int_\Omega(w_+^2 + w_-^2)dxds \\
	&\phantom{xx}{}+ \frac89D_+\int_0^t e^s\int_\Omega
	\big(|\na w_+^{3/2}|^2 + |\na w_-^{3/2}|^2\big)dxds \\
	&\le \frac{4D_-M}{\lambda^2}\int_0^t e^s\int_\Omega(w_+^2 + w_-^2)dxds
  + \frac{4D_-M}{\lambda^2}\int_0^t e^s\int_\Omega dxds.
\end{align*}
By the Poincar\'e inequality, for some constants $C_i>0$, we obtain
$$
  \int_\Omega e^t(w_+^2(t) + w_-^2(t))dx
  \le \int_0^t e^s \int_\Omega\big(-C_1(w_+^3+w_-^3) + C_2(w_+^2+w_-^2)\big)dxds
	+ C_3e^t.
$$
Since $f(x)=-C_1x^3 + C_2x^2 = (-C_1x+C_2)x^2$ has a maximum 
$\widetilde{C}_4>0$ for $x\ge 0$, we can estimate the right-hand side by
$e^t(\widetilde{C}_4\operatorname{meas}(\Omega)+C_3)$. Division by $e^t$ leads to
$$
  \int_\Omega (w_+^2(t) + w_-^2(t))dx \le C_4 
	:=\widetilde{C}_4\operatorname{meas}(\Omega)+C_3,
$$
which does not depend on time. In particular, this shows
that $w_\pm$ is bounded in $L^\infty(0,T;$ $L^1(\Omega))$ uniformly in time.
Thus, by Lemma \ref{lem.iter}, $\|w_\pm(t)\|_{L^\infty(\Omega)}\le K$ for some
constant $K>0$ and $n_\pm(t)\ge \exp(-K-m)$ in $\Omega$, $t>0$. This finishes
the proof.
\end{proof}

\begin{remark}\rm
The factor $e^t$ is necessary to derive time-uniform bounds. Indeed, without this
factor, the last term on the right-hand side of \eqref{3.wq} becomes
$4D_-M\lambda^{-2}\int_0^t\int_\Omega dsdx$ which is unbounded as $t\to\infty$.
\qed
\end{remark}

We introduce the relative free energy
$$
  H(t) = \int_\Omega\bigg(h(n_+|n_\infty) + h(n_-|n_\infty)
	+ \frac{\lambda^2}{2}|\na(V-V_\infty)|^2\bigg)(t)dx,
$$
where $h(n_\pm|n_\infty) = n_\pm\log(2n_\pm/n_\infty) - n_\pm + \tfrac12 n_\infty$,
and the electrochemical potentials $\phi_\pm = \log n_\pm + V$.

\begin{lemma}[Relative free energy estimate]\label{lem.ent}
It holds that
\begin{align*}
  \frac{dH}{dt} &\le -\frac{D_+}{2}\int_\Omega
	\big(n_+|\na(\phi_+ - \phi_\infty)|^2
	+ n_-|\na(\phi_- - \phi_\infty)|^2\big)dx \\
	&\phantom{xx}{}
	+ \frac{D_-}{2}\int_\Omega\big((n_+ - \tfrac12 n_\infty)^2
	+ (n_- - \tfrac12 n_\infty)^2\big)|\na\phi_\infty|^2 dx \\
	&\phantom{xx}{}
	- \frac{1}{8\tau}\int_\Omega\big(\sqrt{n_+}-\sqrt{n_-}\big)^2 dx.
\end{align*}
\end{lemma}

\begin{proof}
Using the Poisson equation and the definitions $\phi_\pm=\log n_\pm+V$ and
$\phi_\infty=\log n_\infty+V_\infty$, it follows that
\begin{align}
  \frac{dH}{dt} &= \bigg\langle\pa_t n_+,\log\frac{2n_+}{n_\infty}\bigg\rangle
	+ \bigg\langle\pa_t n_-,\log\frac{2
		n_-}{n_\infty}\bigg\rangle
	- \lambda^2\langle\pa_t\Delta(V-V_\infty),V-V_\infty\rangle \nonumber \\
	&= \bigg\langle\pa_t n_+,\log\frac{n_+}{n_\infty}+\log 2\bigg\rangle
	+ \bigg\langle\pa_t n_-,\log\frac{n_-}{n_\infty}+\log 2\bigg\rangle
	+ \langle\pa_t(n_+ + n_-),V-V_\infty\rangle \nonumber \\
	&= \langle\pa_t n_+,\phi_+-\phi_\infty+\log 2\rangle
	+ \langle\pa_t n_-,\phi_--\phi_\infty+\log 2\rangle \label{dHdt}
\end{align}
This can be made rigorous similarly as in \cite[formula (5.18)]{Jue95},
together with the techniques in \cite[Theorem 3, p.~287]{Eva02}.

Next, we subtract $\frac 12 D_{\pm}\times$\eqref{2.steady} 
from \eqref{1.nplus} and \eqref{1.nminus}, respectively:
\begin{align*}
  \pa_t n_+ - D_+\diver\big(n_+\na(\phi_+-\phi_\infty)
	+ (n_+ - \tfrac12 n_\infty)\na\phi_\infty\big) = -\frac{1}{2\tau}(n_+ - n_-), \\
  \pa_t n_- - D_-\diver\big(n_-\na(\phi_--\phi_\infty)
	+ (n_- - \tfrac12 n_\infty)\na\phi_\infty\big) = -\frac{1}{2\tau}(n_- - n_+).
\end{align*}
Inserting these equations into \eqref{dHdt}, we find that
\begin{align*}
  \frac{dH}{dt} &= -D_+\int_\Omega n_+|\na(\phi_+ - \phi_\infty)|^2 dx
	- D_-\int_\Omega n_-|\na(\phi_- - \phi_\infty)|^2 dx \\
	&\phantom{xx}{}- D_+\int_\Omega(n_+ - \tfrac12 n_\infty)\na\phi_\infty\cdot
	\na(\phi_+ - \phi_\infty)dx \\
	&\phantom{xx}{}- D_-\int_\Omega(n_- - \tfrac12 n_\infty)\na\phi_\infty\cdot
	\na(\phi_- - \phi_\infty)dx \\
  &\phantom{xx}{}- \frac{1}{2\tau}\int_\Omega(n_+ - n_-)(\log n_+ - \log n_-)dx.
\end{align*}
We use the elementary inequality
\begin{equation}\label{3.elem}
  (y-z)(\log y-\log z) \ge \frac14\big(\sqrt{y}-\sqrt{z}\big)^2\quad\mbox{for }
	y,z>0
\end{equation}
to estimate the last term. Then the Young inequality and the lower bound
$n_\pm\ge m$ lead to
\begin{align*}
  \frac{dH}{dt} &\le -\frac{D_+}{2}\int_\Omega
	n_+|\na(\phi_+ - \phi_\infty)|^2 dx
	- \frac{D_-}{2}\int_\Omega n_-|\na(\phi_- - \phi_\infty)|^2 dx \\
	&\phantom{xx}{}
	+ \frac{D_+}{2m}\int_\Omega(n_+ - \tfrac12 n_\infty)^2|\na\phi_\infty|^2 dx
	+ \frac{D_-}{2m}\int_\Omega(n_- - \tfrac12 n_\infty)^2|\na\phi_\infty|^2 dx \\
	&\phantom{xx}{}
	-\frac{1}{8\tau}\int_\Omega\big(\sqrt{n_+}-\sqrt{n_-}\big)^2 dx,
\end{align*}
finishing the proof.
\end{proof}

\begin{lemma}[Lower bound for the chemical potentials]\label{lem.lowerphi}
It holds that
\begin{align*}
  \|\na&(\phi_+-\phi_\infty)\|_{L^2(\Omega)}^2
	+ \|\na(\phi_--\phi_\infty)\|_{L^2(\Omega)}^2 \\
	&\ge C\big(\|n_+-\tfrac12 n_\infty\|_{L^2(\Omega)}^2
	+ \|n_--\tfrac12 n_\infty\|_{L^2(\Omega)}^2
	+ \|\na(V-V_\infty)\|_{L^2(\Omega)}^2\big),
\end{align*}
where $C>0$ depends only on $M$, $M_\infty$, $\lambda$, and $\Omega$.
\end{lemma}

Recall that $M$ is the upper bound for $n_\pm$ (see Lemma \ref{lem.upper})
and $M_\infty$ is the upper bound for $n_\infty$ (see \eqref{2.minfty}).

\begin{proof}
It holds that $\phi_\pm-\phi_\infty+\log 2 =  0$ on $\pa\Omega$.
Thus, the Young and Poincar\'e inequalities yield for any $\eps>0$,
\begin{align}
  \int_\Omega(n_\pm - \tfrac12 n_\infty)(\phi_\pm - \phi_\infty + \log 2)dx
	&\le \eps\|n_\pm - \tfrac12 n_\infty\|_{L^2(\Omega)}^2
	+ C(\eps)\|\phi_\pm - \phi_\infty + \log 2\|_{L^2(\Omega)}^2 \nonumber \\
	&\le \eps\|n_\pm - \tfrac12 n_\infty\|_{L^2(\Omega)}^2
	+ C(\eps,\Omega)\|\na(\phi_\pm - \phi_\infty)\|_{L^2(\Omega)}^2. \label{aux1}
\end{align}
Inserting the definitions of $\phi_\pm$ and $\phi_\infty$, taking into account
the Poisson equation $-\lambda^2\Delta(V-V_\infty)=n_0 - n_\infty$ and
inequality \eqref{3.elem}, and finally
using the bounds $m\le n_\pm\le M$ and $m_\infty\le n_\infty\le M_\infty$, we obtain
\begin{align*}
  \int_\Omega&(n_+-\tfrac12 n_\infty)(\phi_+-\phi_\infty+\log 2)dx
	+ \int_\Omega(n_--\tfrac12 n_\infty)(\phi_--\phi_\infty+\log 2)dx \\
	&= \int_\Omega(n_+-\tfrac12 n_\infty)(\log n_+-\log(\tfrac12 n_\infty))dx
	+ \int_\Omega(n_--\tfrac12 n_\infty)(\log n_--\log(\tfrac12 n_\infty))dx \\
	&\phantom{xx}{}+ \int_\Omega(n_0 - n_\infty)(V-V_\infty)dx \\
  &\ge \frac14\int_\Omega\Big(\sqrt{n_+}-\sqrt{\tfrac12 n_\infty}\Big)^2dx
	+ \frac14\int_\Omega\Big(\sqrt{n_-}-\sqrt{\tfrac12 n_\infty}\Big)^2dx
	+ \lambda^2\int_\Omega|\na(V-V_\infty)|^2dx \\
  &= \frac14\int_\Omega\frac{(n_+-n_\infty/2)^2}
   	{\big(\sqrt{n_+}+\sqrt{n_\infty/2}\big)^2}dx
	+ \frac14\int_\Omega\frac{(n_--n_\infty/2)^2}
	{\big(\sqrt{n_-}+\sqrt{n_\infty/2}\big)^2}dx
	+ \lambda^2\int_\Omega|\na(V-V_\infty)|^2dx \\
	&\ge C_2\int_\Omega(n_+-\tfrac12 n_\infty)^2 dx
	+ C_2\int_\Omega(n_--\tfrac12 n_\infty)^2 dx
	+ \lambda^2\int_\Omega|\na(V-V_\infty)|^2dx,
\end{align*}
where $C_2=\frac14(\sqrt{M}+\sqrt{M_\infty/2})^{-2}$.
Combining this estimate with \eqref{aux1} and taking $\eps<C_1$, we conclude
the proof.
\end{proof}

\begin{lemma}[Bounds for the relative free energy]\label{lem.boundent}
There exist constants $C_\phi$, $C_H>0$ independent of the solution and time such that
\begin{align*}
  H &\le C_\phi\big(\|\na(\phi_+-\phi_\infty)\|_{L^2(\Omega)}^2
	+ \|\na(\phi_--\phi_\infty)\|_{L^2(\Omega)}^2\big), \\
  H &\ge C_H\big(\|n_+-\tfrac12 n_\infty\|_{L^2(\Omega)}^2
	+ \|n_--\tfrac12 n_\infty\|_{L^2(\Omega)}^2\big).
\end{align*}
\end{lemma}

\begin{proof}
Set $f(y)=y\log(y/z)-y+z$ for some fixed $z>0$. A Taylor expansion shows that
$$
  y\log\frac{y}{z} - y + z = f(y) = f(z) + f'(z)(y-z) + \frac12 f''(\xi)(y-z)^2
	= \frac{(y-z)^2}{2\xi},
$$
where $\xi$ is between $y$ and $z$. Consequently, since $n_\pm\ge m$ and
$n_\infty\ge m_\infty$,
$$
  n_\pm\log\frac{2n_\pm}{n_\infty} - n_\pm + \frac12 n_\infty
	\le \frac{1}{2C_1}(n_\pm-\tfrac12 n_\infty)^2,
$$
where $C_1 = \min\{m,m_\infty/2\}$,
and, using Lemma \ref{lem.lowerphi}, we find that
\begin{align*}
  H &\le \max\bigg\{\frac{1}{2C_1},\frac{\lambda^2}{2}\bigg\}
	\big(\|n_+-\tfrac12 n_\infty\|_{L^2(\Omega)}^2
	+ \|n_--\tfrac12 n_\infty\|_{L^2(\Omega)}^2
	+ \|\na(V-V_\infty)\|_{L^2(\Omega)}^2\big) \\
	&\le \max\bigg\{\frac{1}{2C_1},\frac{\lambda^2}{2}\bigg\}C\big(\|\na(\phi_+-\phi_\infty)\|_{L^2(\Omega)}^2
	+ \|\na(\phi_--\phi_\infty)\|_{L^2(\Omega)}^2\big).
\end{align*}
For the second estimate, we use
$$
  n_\pm\log\frac{2n_\pm}{n_\infty} - n_\pm + \frac12 n_\infty
	\ge \frac{(n_\pm-\tfrac12 n_\infty)^2}{2C_2}, \quad\mbox{where }
	2C_2 = \max\{M,M_\infty/2\},
$$
to conclude that
$$
  H \ge \frac{1}{2C_2}\int_\Omega
	\big((n_+-\tfrac12 n_\infty)^2 + (n_--\tfrac12 n_\infty)^2\big)dx.
$$
This finishes the proof.
\end{proof}


\section{Proof of Theorem \ref{thm.time}}\label{sec.proof}

The starting point is the free-energy inequality in Lemma \ref{lem.ent}.
We need to estimate the integral containing $\na\phi_\infty$. In view of
Lemmas \ref{lem.stat} and \ref{lem.boundent},
\begin{align}
  \int_\Omega & (n_\pm-\tfrac12 n_\infty)^2|\na\phi_\infty|^2 dx
	\le\|\na\phi_\infty\|_{L^\infty(\Omega)}^2\|n_\pm-\tfrac12 n_\infty\|_{L^2(\Omega)}^2
	\nonumber \\
	&\le C_\infty^2\|\phi_D\|_{W^{2,q_0}(\Omega)}^2
	\|n_\pm-\tfrac12 n_\infty\|_{L^2(\Omega)}^2
	\le C_\infty^2 C_H^{-1}\|\phi_D\|_{W^{2,q_0}(\Omega)}^2 H. \label{delta}
\end{align}
By the lower bound of $n_\pm$ and Lemma \ref{lem.lowerphi}, the free-energy inequality
in Lemma \ref{lem.ent} becomes
$$
  \frac{dH}{dt} + \bigg(\frac{D_+m}{2C_\phi} - \frac{D_-C^2_\infty}{C_H}
	\|\phi_D\|_{W^{2,q_0}(\Omega)}^2\bigg)H
	+ \frac{1}{8\tau}\|\sqrt{n_+}-\sqrt{n_-}\|_{L^2(\Omega)}^2 \le 0.
$$
Let $\delta>0$ satisfy $2\kappa:= D_+m/(2C_\phi) - D_-(C^2_\infty/C_H)\delta^2>0$
and choose $n_D$ and $V_D$ such that$\|\phi_D\|_{W^{2,q_0}(\Omega)}\le \delta$.
Then Gronwall's lemma implies that
$H(t)\le H(0)\exp(-2\kappa t)$ for $t>0$. By Lemma \ref{lem.boundent},
$$
  \|n_+(t)-\tfrac12 n_\infty\|_{L^2(\Omega)}
	+ \|n_-(t)-\tfrac12 n_\infty\|_{L^2(\Omega)}
	\le C_H^{-1/2} H(0)^{1/2}e^{-\kappa t}, \quad t>0.
$$
The $H^1(\Omega)$ elliptic estimate for the Poisson problem
$-\lambda^2\Delta (V-V_\infty)=(n_+-\frac12 n_\infty)+(n_--\frac12 n_\infty)$
in $\Omega$, $V-V_\infty=0$ on $\pa\Omega$ gives
$$
  \|V(t)-V_\infty\|_{H^1(\Omega)}
	\le C\|n_+(t)-\tfrac12 n_\infty\|_{L^2(\Omega)}
	+ C\|n_-(t)-\tfrac12 n_\infty\|_{L^2(\Omega)}
	\le Ce^{-\kappa t},
$$
which proves the first estimate. For the second result, recall that
we can decompose $N(t)$ as $N(t)=\tfrac12(n_++n_-)\sigma_0+(\vec n_\perp+\tfrac12
(n_+-n_-)\vec\mu)\cdot\vec \sigma$. We use $\vec n_\perp$ as a test function in 
\eqref{1.nperp}:
\begin{align*}
  \frac12\frac{d}{dt}&\int_\Omega|\vec n_\perp|^2 dx
	+ \frac{D}{\eta}\int_\Omega|\na\vec n_\perp|^2 dx
	+ \frac{1}{\tau}\int_\Omega|\vec n_\perp|^2 dx \\
	&= -\frac{D}{2\eta}\int_\Omega\na|\vec n_\perp|^2\cdot\na V dx
	= -\frac{D}{2\eta\lambda^2}\int_\Omega |\vec n_\perp|^2(n_0-g(x))dx \\
	&\le \frac{D}{2\eta\lambda^2}\|g\|_{L^\infty(\Omega)}\int_\Omega|\vec n_\perp|^2dx.
\end{align*}
Thus, if $\tau\le 2\eta\lambda^2/(D\|g\|_{L^\infty(\Omega)})$, the Poincar\'e
inequality shows that
$$
  \frac{d}{dt}\int_\Omega|\vec n_\perp|^2 dx
	+ 2C(D,\eta, \Omega)\int_\Omega|\vec n_\perp|^2 dx \le 0.
$$
By Gronwall's lemma,
$$
  \|\vec n_\perp(t)\|_{L^2(\Omega)}
	\le \|\vec n_\perp(0)\|_{L^2(\Omega)}e^{-C(D,\eta,\Omega)t}, \quad t>0.
$$
Therefore, we find that
\begin{align*}
  \big\|N(t)-\tfrac12 n_\infty\sigma_0\big\|_{L^2(\Omega;\C^{2\times 2})}
	&\le \big\|(\tfrac12(n_+-\tfrac12n_\infty)+\tfrac12(n_--\tfrac12n_\infty))
	\sigma_0\big\|_{L^2(\Omega;\C^{2\times 2})} \\
  &\phantom{xx}{}+ \big\|\vec n_\perp+\tfrac12(n_+-n_-)\vec\mu)\cdot\vec \sigma
	\big\|_{L^2(\Omega;\C^{2\times 2})}\\
  &\le C_H^{-1/2} H(0)^{1/2}e^{-\kappa t} 
	+  \|\vec n_\perp(0)\|_{L^2(\Omega)}e^{-C(D,\eta,\Omega)t}\le C_0^*e^{-\kappa^*t},
\end{align*}
where $C_0^*=\max({2C_H^{-1/2} H(0)^{1/2},\|\vec n_\perp(0)\|_{L^2(\Omega)}})$ 
and $\kappa^*=\min({\kappa,C(D,\eta,\Omega)t})$.
This concludes the proof of Theorem \ref{thm.time}.


\begin{remark}[Special initial spin-vector density]\label{rem.tau}\rm
Let $\m =(0,0,1)^\top$ and $\vec n^0=(0,0,n_3^0)$. Then the components
$n_1$ and $n_2$ of the spin-vector density satisfy the equation
$$
  \pa_t n_i = \diver\bigg(\frac{D}{\eta^2}(\na n_i+n_i\na V)\bigg), \quad i=1,2,
$$
with boundary conditions $n_i=0$ on $\pa\Omega$ and initial conditions
$n_i(0)=0$ in $\Omega$. The unique solution is given by $n_i(t)=0$ for all $t>0$
and $i=1,2$. Since $\vec n = \vec n_\perp + (\vec n\cdot\m )\m
= \vec n_\perp + n_3\m $, the perpendicular component vanishes, $\vec n_\perp(t)=0$.
We conclude that the dynamics of the system is completely determined by
$n_\pm$, and the proof of Theorem \ref{thm.time} gives the exponential decay
without any condition on $\tau$. In particular, the density matrix
$N(t)=\frac12(n_+ + n_-)\sigma_0 + \frac12(n_+ - n_-)\sigma_3$ converges exponentially
fast towards $\frac12 n_\infty\sigma_0$ as $t\to\infty$.
\qed
\end{remark}

\begin{remark}[Smallness condition on the relaxation time]\label{rem.smalltau}\rm\
We discuss the physical rel\-e\-vance of the smallness condition of the
scaled relaxation time $\tau\le 2\eta\lambda^2/(D\|g\|_{L^\infty(\Omega)})$.
In scaled variables, we may assume that $D=1$ and $\|g\|_{L^\infty(\Omega)}=1$.
The scaled Debye length is given by $\lambda^2 = \eps_sU_T/(qL^2g^*)
= 2.7\cdot 10^{-1}$,
where the physical parameters are explained in Table \ref{table}.
We have assumed that the semiconductor material is lowly doped.
The scaled relaxation time
is $\tau = \tau_0/t^*$, where the typical time is defined by
$t^* = L^2/(\mu_0 U_T) = 1.5\cdot 10^{-11}$\,s.
The spin-flip relaxation time is assumed to be $\tau_0=1$\,ps.
This value is realistic in GaAs quantum wells at temperature $T=50$\,K; see
\cite[Figure 1]{BaFe92}. It follows that $\tau=6\cdot 10^{-2}$.
Thus, the inequality $\tau\le 2\eta\lambda^2/(D\|g\|_{L^\infty(\Omega)})$
is satisfied if $\eta\ge 0.11$ or $p\le 0.99$. This covers almost the full
range of $p\in[0,1)$.
\qed
\end{remark}

\begin{table}
\begin{tabular}{|l|l|l|}
\hline
Parameter & physical meaning & numerical value \\ \hline
$q$ & elementary charge & $1.6\cdot 10^{-19}$\,As \\
$\eps_s$ & permittivity constant & $10^{-12}$\,As/(Vcm) \\
$\mu_0$ & (low field) mobility constant & $1.5\cdot 10^3$\, cm$^2$/(Vs) \\
$U_T$ & thermal voltage at $T=50$\,K & $4.3\cdot 10^{-3}$\,V \\
$g^*$ & maximal doping concentration & $10^{15}$/cm$^3$ \\
$\tau_0$ & spin-flip relaxation time & $10^{-12}$\,s \\
$L$ & length of the device & $10^{-5}$\,cm \\
\hline
\end{tabular}
\caption{Physical parameters.}\label{table}
\end{table}


\begin{appendix}
\section{A boundedness result}

The following lemma is an extension of a result due to Kowalczyk \cite{Kow05},
based on an iteration technique \cite{Ali79}. It slightly generalizes
\cite[Lemma A.1]{JNS15}. Although the result should be known to experts,
we present a proof for completeness.

\begin{lemma}[Boundedness from iteration]\label{lem.iter}
Let $\Omega\subset\R^d$ ($d\ge 1$) be a bounded domain and let
$w_i^{q/2}\in L^2(0,T;H^1(\Omega))\cap C^0([0,T];L^2(\Omega))$
for all $q\in\N$ with $q\ge 2$ with $w_i\ge 0$,
$w_i=0$ on $\pa\Omega$, and $w_i(0)=0$ in $\Omega$ for $i=1,\ldots,n$.
Assume that there are
constants $K_0$, $K_1$, $K_2>0$ and $\alpha$, $\beta\ge 0$
such that for all $q\ge 2$, $t>0$,
\begin{align}
  \int_\Omega & e^{t}\sum_{i=1}^n w_i(t)^q dx
	+ K_0\int_0^t\int_\Omega e^{s}\sum_{i=1}^n|\na w_i^{q/2}|^2 dxds \nonumber \\
	&\le K_1 q^\alpha\int_0^t\int_\Omega e^{s}\sum_{i=1}^n w_i^q dxds
	+ K_2 q^\beta e^{t}. \label{a.1}
\end{align}
Then
$$
  w_i(t) \le K = K_3\bigg(\sum_{i=1}^n
	\|w_i\|_{L^\infty(0,\infty;L^1(\Omega))} + 1\bigg)\quad\mbox{in }\Omega,\ t>0,
$$
where $K_3$ depends only on $\alpha$, $\beta$, $d$, $\Omega$, $K_0$, $K_1$, and $K_2$.
\end{lemma}

\begin{proof}
We apply the Gagliardo--Nirenberg inequality \cite[p.~1034]{Zei90}
with $\theta=d/(d+2)<1$ and the
Poincar\'e inequality to deal with the integral over $w_i^q$ on the right-hand side
of \eqref{a.1}:
\begin{align*}
  \int_\Omega w_i^q dx &= \|w_i^{q/2}\|_{L^2(\Omega)}^2
	\le C_1\|\na w_i^{q/2}\|_{L^2(\Omega)}^{2\theta}
	\|w_i^{q/2}\|_{L^1(\Omega)}^{2(1-\theta)} \\
	&\le \eps\|\na w_i^{q/2}\|_{L^2(\Omega)}^2
	+ C_1^{1+d/2}\eps^{-d/2}\|w_i^{q/2}\|_{L^1(\Omega)}^2
\end{align*}
for any $\eps>0$. Choosing $\eps = K_0/(2q^\alpha K_1)$, which is equivalent to
$K_1 q^\alpha\eps = K_0/2$, \eqref{a.1} becomes
\begin{align*}
  \int_\Omega &e^{t}\sum_{i=1}^n w_i(t)^q dx
	+ \frac{K_0}{2}\int_0^t\int_\Omega e^{s}\sum_{i=1}^n |\na w_i^{q/2}|^2 dxds \\
	&\le C_2 q^{\alpha(1+d/2)}
	\int_0^t e^{s}\sum_{i=1}^n\|w_i\|_{L^{q/2}(\Omega)}^q ds + K_2 q^\beta e^{t},
\end{align*}
where $C_2$ depends on $d$, $K_0$, and $K_1$. We obtain
$$
  \sum_{i=1}^n \|w_i(t)\|_{L^q(\Omega)}^q \le C_2q^{\alpha(1+d/2)}\int_0^t e^{-(t-s)}
	\sum_{i=1}^n \|w_i(s)\|_{L^{q/2}(\Omega)}^q ds + K_2 q^\beta
$$
and, taking the supremum over time,
\begin{align*}
  \sup_{0<s<t}\sum_{i=1}^n \|w_i(s)\|_{L^q(\Omega)}^q
	&\le C_2q^{\alpha(1+d/2)}(1-e^{-t})\sup_{0<s<t}
	\sum_{i=1}^n\|w_i(s)\|_{L^{q/2}(\Omega)}^q + K_2 q^\beta \\
	&\le C_2q^{\alpha(1+d/2)}\sup_{0<s<t}\sum_{i=1}^n \|w_i(s)\|_{L^{q/2}(\Omega)}^q
	+ K_2 q^\beta.
\end{align*}

We choose $q=2^k$ for $k\ge 0$ and set
$b_k = \sum_{i=1}^n \|w_i\|_{L^\infty(0,T;L^{2^k}(\Omega))}^{2^k}+1$. Then
\begin{align*}
  b_k &\le C_2 2^{\alpha(1+d/2)k}\sum_{i=1}^n\|w_i\|_{L^\infty(0,T;
	L^{2^{k-1}}(\Omega))}^{2^{k}} + (K_2+1) 2^{\beta k} \\
	&\le \max\big\{C_2 2^{\alpha(1+d/2)}, (K_2+1)^{1/k} 2^\beta\big\}^k
	\bigg(\sum_{i=1}^n\|w_i\|_{L^\infty(0,T;L^{2^{k-1}}(\Omega))}^{2^{k}}+1\bigg) \\
	&\le \max\big\{C_2 2^{\alpha(1+d/2)}, (K_2+1)2^\beta\big\}^k
	\bigg(\sum_{i=1}^n\|w_i\|_{L^\infty(0,T;L^{2^{k-1}}(\Omega))}^{2^{k-1}}+1\bigg)^2 \\
	&= \gamma^k b_{k-1}^2,
\end{align*}
where
$$
	\gamma = \max\big\{C_2 2^{\alpha(1+d/2)}, (K_2+1)2^\beta\big\}.
$$
To solve this recursion, we set $c_k = \gamma^{k+2}b_k$. Then
$$
  c_k \le \gamma^{2(k+1)}b_{k-1}^2 = (\gamma^{k+1}b_{k-1})^2 = c_{k-1}^2,
$$
which gives $c_k\le c_0^{2^k}\le \gamma^{2^{k+1}}b_0^{2^k}$. Consequently,
$b_k = \gamma^{-k-2}c_k\le \gamma^{2^{k+1}-k-2}b_0^{2^k}$ and,
after taking the $2^k$th root,
$$
  \|w_i\|_{L^\infty(0,T;L^{2^k}(\Omega))}
	\le b_k^{2^{-k}}
	\le \gamma^{2-2^{-k}(k+2)}\bigg(\sum_{i=1}^n
	\|w_i\|_{L^\infty(0,T;L^{1}(\Omega))} + 1\bigg).
$$
The limit $k\to\infty$ concludes the proof.
\end{proof}

\end{appendix}



\begin{thebibliography}{11}
\bibitem{ARBVHPS15} C.~Abert, M.~Ruggeri, F.~Bruckner, C.~Vogler, G.~Hrkac,
D.~Praetorius, and D.~Suess. A three-dimensional spin-diffusion
model for micromagnetics. {\em Scientific Reports} 5 (2015), no.~14855, 11 pages.

\bibitem{Ala95} F.~Alabau. New uniqueness theorems for the one-dimensional
drift-diffusion semiconductor device equations.
{\em SIAM J. Math. Anal.} 26 (1995), 715--737.

\bibitem{Ali79} N.~Alikakos. $L^p$ bounds of solutions of reaction-diffusion equations.
{\em Commun. Partial Diff. Eqs.} 4 (1979), 827--868.

\bibitem{AMT00} A.~Arnold, P.~Markowich, and G.~Toscani. On large time asymptotics for
drift-diffusion-Poisson systems. {\em Transp. Theory Stat. Phys.} 29 (2000), 571--581.

\bibitem{BaMe10} L.~Barletti and F.~M\'ehats. Quantum drift-diffusion modeling of
spin transport in nanostructures. {\em J. Math. Phys.} 51 (2010), no.~053304,
20 pages.

\bibitem{BMNP15} L.~Barletti, F.~M\'ehats, C.~Negulescu, and S.~Possanner.
Numerical study of a quantum-diffusive spin model for two-dimensional
electron gas. {\em Commun. Math. Sci.} 13 (2015), 1347--1378.

\bibitem{BaFe92} G.~Bastard and R.~Ferreira. Spin-flip scattering times in
semiconductor quantum wells. {\em Surface Science} 267 (1992), 335--341.

\bibitem{BMV04} N.~Ben Abdallah, F.~M\'ehats, and N.~Vauchelet. A note on the long
time behavior for the drift-diffusion-Poisson system.
{\em C. R. Acad. Sci. Paris, S\'er. I} 339 (2004), 683--688.

\bibitem{BiDo00} P.~Biler and J.~Dolbeault. Long time behavior of solutions to
Nernst--Planck and Debye--H\"uckel drift-diffusion systems.
{\em Ann. Henri Poincar\'e} 1 (2000), 461--472.

\bibitem{BBA10} L.~Bonilla, L.~Barletti, and M.~Alvaro. Nonlinear electron and spin
transport in semiconductor superlattices. {\em Progress in Industrial Mathematics
at ECMI 2008}, pages 141--146, Springer, Heidelberg, 2010.

\bibitem{CJS16} C.~Chainais-Hillairet, A.~J\"ungel, and P.~Shpartko. A
finite-volume scheme for a spinorial matrix drift-diffusion model for
semiconductors. {\em Numer. Meth. Partial Diff. Eqs.} 32 (2016), 819--846.

\bibitem{DiWu08} M.~Di Francesco and M.~Wunsch. Large time behavior in Wasserstein
spaces and relative entropy for bipolar drift-diffusion-Poisson models.
{\em Monatsh. Math.} 154 (2008), 39--50.

\bibitem{ElH14} R.~El Hajj. Diffusion models for spin transport derived from the
spinor Boltzmann equation. {\em Commun. Math. Sci.} 12 (2014), 565--592.

\bibitem{Eva02} L.~Evans. {\em Partial Differential Equations}. Amer. Math. Soc.,
Providence, USA, 2002.

\bibitem{FMESZ07} J.~Fabian, A.~Matos-Abiague, C.~Ertler, P.~Stamo and I.~\v{Z}uti\'c.
Semiconductor spintronics. {\em Acta Physica Slovaca} 57 (2007), 565--907.

\bibitem{Gaj85} H.~Gajewski. On existence, uniqueness and asymptotic behavior of
solutions of the basic equations for carrier transport in semiconductors.
{\em Z. Angew. Math. Mech.} 65 (1985), 101--108.

\bibitem{GaGr86} H.~Gajewski and K.~Gr\"oger. On the basic equations for carrier
transport in semiconductors. {\em J. Math. Anal. Appl.} 113 (1986), 12--35.

\bibitem{GaGr89} H.~Gajewski and K.~Gr\"oger. Semiconductor equations for variable
mobilities based on Boltzmann statistics or Fermi--Dirac statistics.
{\em Math. Nachr.} 140 (1989), 7--36.

\bibitem{GaWa07} C.~Garc\'{\i}a-Cervera and X.-P.~Wang. Spin-polarized transport:
Existence of weak solutions. {\em Discrete Contin. Dyn. Sys. Ser. B} 7 (2007), 87--100.

\bibitem{GaWa15} C.~Garc\'{\i}a-Cervera and X.-P.~Wang. A note on 'Spin-polarized
transport: existence of weak solutions'.
{\em Discrete Contin. Dyn. Sys. Ser. B} 20 (2015), 2761--2763.

\bibitem{Gli08} A.~Glitzky. Analysis of a spin-polarized drift-diffusion model.
{\em Adv. Math. Sci. Appl.} 18 (2008), 401--427.

\bibitem{GlGa10} A.~Glitzky and K.~G\"artner. Existence of bounded steady state
solutions to spin-polarized drift-diffusion systems.
{\em SIAM J. Math. Anal.} 41 (2010), 2489--2513.

\bibitem{Jue95} A.~J\"ungel. Qualitative behavior of solutions of a degenerate
nonlinear drift-diffusion model for semiconductors.
{\em Math. Models Meth. Appl. Sci.}5 (1995), 497--518.

\bibitem{Jue09} A.~J\"ungel. {\em Transport Equations for Semiconductors}.
Lecture Notes in Physics, Vol.~773. Springer, Berlin, 2009.

\bibitem{JNS15} A.~J\"ungel, C.~Negulescu, and P.~Shpartko. Bounded weak solutions to
a matrix drift-diffusion model for spin-coherent electron transport in semiconductors.
{\em Math. Models Meth. Appl. Sci.} 25 (2015), 929--958.

\bibitem{KoKa08} R.~Kobayashi and S.~Kawashima. Decay estimates and large time
behavior of solutions to the drift-diffusion system.
{\em Funkcialaj Ekvacioj} 51 (2008), 371--394.

\bibitem{Kow05} R.~Kowalczyk. Preventing blow-up in a chemotaxis model.
{\em J. Math. Anal. Appl.} 305 (2005), 566--588.

\bibitem{Mar86} P.~Markowich. {\em The Stationary Semiconductor Device Equations}.
Springer, Wien, 1986.


\bibitem{Moc74} M.~Mock. An initial value problem from semiconductor device theory.
{\em SIAM J. Math. Anal.} 5 (1974), 597--612.

\bibitem{Moc75} M.~Mock. Asymptotic behavior of solutions of transport equations for
semiconductor devices. {\em J. Math. Anal. Appl.} 49 (1975), 215--225.

\bibitem{Moc82} M.~Mock. An example of nonuniqueness of stationary solutions in
semiconductor device models. {\em Compel} 1 (1982), 165--174.

\bibitem{OgYa09} T.~Ogawa and M.~Yamamoto. Asymptotic behavior of solutions to
drift-diffusion system with generalized dissipation.
{\em Math. Models Meth. Appl. Sci.} 19 (2009), 939--967.

\bibitem{PoNe11} S.~Possanner and C.~Negulescu. Diffusion limit of a generalized
matrix Boltzmann equation for spin-polarized transport.
{\em Kinetic Related Models} 4 (2011), 1159--1191.

\bibitem{Pou91} F.~Poupaud. Diffusion approximation of the linear semiconductor
Boltzmann equation: Analysis of boundary layers. {\em Asympt. Anal.} 4 (1991), 293--317.

\bibitem{PuGu10} X.~Pu and B.~Gu. Global smooth solutions for the one-dimensional
spin-polarized transport equation. {\em Nonlin. Anal.} 72 (2010), 1481--1487.

\bibitem{RAHSP16} M.~Ruggeri, C.~Abert, G.~Hrkac, D.~Suess, and D.~Praetorius.
Coupling of dynamical micromagnetism and a stationary spin drift-diffusion equation:
A step towards a fully self-consistent spintronics framework.
{\em Physica B: Condens. Matter} 486 (2016), 88--91.

\bibitem{Van50} W.~Van Roosbroeck. Theory of flow of electrons and holes in
Germanium and other semiconductors. {\em Bell System Tech.} 29 (1950), 560--607.

\bibitem{Yam12} M.~Yamamoto. Large-time behavior of solutions to the drift-diffusion
equation with fractional dissipation. {\em Diff. Int. Eqs.} 25 (2012), 731--758.

\bibitem{Zam14} N.~Zamponi. Analysis of a drift-diffusion model with velocity
saturation for spin-polarized transport in semiconductors.
{\em J. Math. Anal. Appl.} 420 (2014), 1167--1181.

\bibitem{ZaJu13} N.~Zamponi and A.~J\"ungel. Two spinorial drift-diffusion models for
quantum electron transport in graphene. {\em Commun. Math. Sci.} 11 (2013), 927--950.

\bibitem{ZaJu16} N.~Zamponi and A.~J\"ungel. Analysis of a coupled spin
drift-diffusion Maxwell--Landau--Lifshitz system.
{\em J. Diff. Eqs.} 260 (2016), 6828--6854.

\bibitem{Zei90} E.~Zeidler. {\em Nonlinear Functional Analysis and Its Applications}. 
Volume II/B. Springer, New York, 1990.

\end{thebibliography}
\end{document}